\documentclass{article}
\usepackage{amssymb, amsmath, amsthm}
\theoremstyle{plain}
\newtheorem{theorem}{Theorem}
\newtheorem{lemma}[theorem]{Lemma}
\newtheorem{proposition}[theorem]{Proposition}
\newtheorem{claim}[theorem]{Claim}

\newtheorem*{theorem*}{Theorem}
\newtheorem*{lemma*}{Lemma}

\theoremstyle{definition}

\newtheorem*{definition*}{Definition}

\newtheorem*{example*}{Example}

\theoremstyle{remark}

\newtheorem*{remark*}{Remark}

\newcommand{\comment}[1]{}
\newcommand{\edge}[2]{\ensuremath{\left\{#1,#2\right\}}}
\newcommand{\dedge}[2]{\ensuremath{\left(#1,#2\right)}}
\newcommand{\E}{\ensuremath{E}}

\makeatletter
\def\imod#1{\allowbreak\mkern5.5mu({\operator@font mod}\,\,#1)}
\makeatother

\title{The number of $F$-matchings in almost every tree is a zero residue}
\date{}
\author{
Noga Alon
\thanks{
Sackler School of Mathematics and Blavatnik School of Computer
Science, Raymond and Beverly Sackler Faculty of Exact Sciences,
Tel Aviv University, Tel Aviv, 69978, Israel. Email:
{\tt{nogaa@tau.ac.il}}. Research supported in part by an ERC
advanced grant, by a USA-Israeli BSF grant
and by the Hermann Minkowski Minerva Center for
Geometry at Tel Aviv University. } \and Simi Haber
\thanks{Sackler School of Mathematical Sciences, Raymond and Beverly Sackler
Faculty of Exact Sciences, Tel Aviv University, Tel Aviv, 69978,
Israel. E-mail: {\tt{habbersi@post.tau.ac.il}}.}
 \and Michael Krivelevich
\thanks{Sackler School of Mathematical Sciences, Raymond and Beverly Sackler
Faculty of Exact Sciences, Tel Aviv University, Tel Aviv, 69978,
Israel. E-mail: {\tt{krivelev@post.tau.ac.il}}. Research supported
in part by USA-Israel BSF Grant 2006322, by grant 1063/08 from the
Israel Science Foundation, and by a Pazy memorial award.} 
}

\begin{document}

\maketitle

\begin{abstract}
For graphs $F$ and $G$ an \emph{$F$-matching} in $G$ is a subgraph
of $G$ consisting of pairwise vertex disjoint copies of $F$. The
number of $F$-matchings in $G$ is denoted by $s(F,G)$. We show that
for every fixed positive integer $m$ and every fixed tree $F$, the
probability that $s(F,\mathcal{T}_n) \equiv 0 \imod{m}$, where
$\mathcal{T}_n$ is a random labeled tree with $n$ vertices, tends to
one exponentially fast as $n$ grows to infinity. A similar result is
proven for induced $F$-matchings. This generalizes a recent result
of Wagner who showed that the number of independent sets in a random
labeled tree is almost surely a zero residue.
\end{abstract}

\section{Introduction}
The number of independent sets in graphs is an important counting
parameter. It is particularly well-studied for trees and tree-like
structures.  Prodinger and Tichy showed in
\cite{A:prodinger_&_tichy1982} that the star and the path maximize
and  minimize, respectively, the number of independent sets among
all trees of a given size. Part of the interest in this graph
invariant stems from the fact that the number of independent sets
plays a role in statistical physics as well as in mathematical
chemistry, where it is known as the \emph{Merrifield-Simmons index}
\cite{B:merrifield_&_simmons1989}. A problem that arises in this
context is the inverse problem: determine a graph within a given
class of graphs (such as the class of all trees) with a given number
of independent sets. It is an open conjecture \cite{A:linek1989}
(see also \cite{A:li_&_li_&_wang2003}) that all but finitely many
positive integers can be represented as the number of independent
sets of some tree. Recently Wagner \cite{A:wagner2009} published a
surprising result that may partially explain why the inverse problem
for independent sets in trees is difficult. He showed that for every
positive integer $m$, the number of independent sets in a random
tree with $n$ vertices is zero modulo $m$ with probability
exponentially close to one. Wagner's proof does not give an
intuitive explanation of the aforementioned fact. In this paper we
give a probabilistic proof for Wagner's result. Our proof is
intuitive and simple, thus allowing us to generalize the result
significantly. We refer the reader to \cite{A:wagner2009} for
further motivation and for a recent survey of previous results
regarding the number of independent sets in trees.

Another graph parameter popular in statistical physics and in
mathematical
chemistry is the \emph{Hosoya index} which is the number of
matchings in the graph. While the inverse problem for the number of
matchings in trees is easy, as the star with $n$ vertices has
exactly $n$ matchings, finding the distribution of this number is
still open, as is the case with
the number of independent sets. Wagner mentions
in \cite{A:wagner2009} that his method could be applied to the
number of matchings as well, showing that asymptotically this number
is typically
divisible by any constant $m$. This may serve as an explanation for
the hardness of obtaining distribution results.

Both independent sets and matchings are special cases of
$F$-matchings. Let $F$ and $G$ be graphs. An \emph{$F$-matching} in
$G$ is a subgraph of $G$ consisting of pairwise vertex disjoint
copies of $F$. We say that the $F$-matching is \emph{induced} in $G$
if no additional edge of $G$ is spanned by the vertices of $G$
covered by the matching. These two closely related notions
generalize naturally matchings and independent sets. Indeed, if $F$
is the graph with two vertices and one edge then an $F$-matching is
simply a matching. If $F$ is a single vertex then an induced
$F$-matching is an independent set.

Given graphs $F$ and $G$ we denote the set of $F$-matchings in $G$
by $S(F,G)$ and its size by $s(F,G)$. The set of all induced
$F$-matchings in $G$ is denoted by $S'(F,G)$ with $s'(F,G) =
|S'(F,G)|$ being its size.

In this paper $G$ will be drawn at random from a probability space
of graphs. We define the \emph{random tree} $\mathcal{T}_n$ to be
the set of all $n^{n-2}$ labeled trees on $n$ vertices endowed with
the uniform distribution.

Our main results are the following:
\begin{theorem} \label{Thm: main thm, general}
Let $F$ be a tree that is not a single vertex and let $m$ be a
positive integer. Then there is a constant $c=c(F,m)>0$ such that
the number of $F$-matchings in the random tree $\mathcal{T}_n$ is
zero modulo $m$ with probability at least $1-e^{-cn}$.
\end{theorem}
Note that when $F$ is a single vertex, the number of $F$-matchings in any
graph with $n$ vertices is $2^n$.

\begin{theorem} \label{Thm: main thm, induced}
Let $F$ be a tree and let $m$ be a positive integer. Then there is a
constant $c'=c'(F,m)>0$ such that the number of induced
$F$-matchings in the random tree $\mathcal{T}_n$ is  zero
modulo $m$ with probability at least $1-e^{-c'n}$.
\end{theorem}

Wagner's result is an immediate consequence of Theorem \ref{Thm:
main thm, induced} --- simply take $F$ to be a single vertex.

In the next section we prove Theorem \ref{Thm: main thm, general},
in Section \ref{sec: induced case} we describe a similar proof of
the induced case and in the last section we state some extensions
and conclude with a few remarks and open questions. Our extensions
include the fact that the assertions of both theorems hold when the
random tree $\mathcal{T}_n$ is replaced by a random planar graph on
$n$ vertices.

\section{The non-induced case}\label{sec: general case}
In this section we prove Theorem \ref{Thm: main thm, general}. The
proof is probabilistic and has two parts, a probabilistic claim
(Lemma \ref{Lem: aas has R cong T_n^(u,v)}) and a deterministic
claim (Lemma \ref{Lem: Nullifying tree, non induced}). Theorem
\ref{Thm: main thm, general} is an immediate consequence of these
claims.

We shall use the following notation. Let $T$ be a tree and assume
that $\edge{u}{v}$ is an edge in $T$. We define a rooted tree
$T^{(u,v)}$ by first setting $v$ as the root --- this defines a
direction of parenthood in $T$
--- and then removing $u$ along with its descendants.
Note that $T^{(u,v)}$ is a rooted (undirected) tree. If $R$ is a
rooted tree isomorphic to $T^{(u,v)}$ (a fact we denote by $R \cong
T^{(u,v)}$) for some edge $\edge{u}{v} \in T$, we say that \emph{$T$
has an $R$-leaf}. The next Lemma states that for every fixed rooted
tree $R$, a random tree has an $R$-leaf with probability
exponentially close to 1.

\begin{lemma} \label{Lem: aas has R cong T_n^(u,v)}
Let $R$ be a rooted tree. There exists a constant $c = c(R) > 0$
such that
\[ \Pr [\exists \edge{u}{v} \in \mathcal{T}_n
\text{ s.t.\ } R \cong \mathcal{T}_n^{(u,v)}] > 1-e^{-cn} . \]
\end{lemma}
\begin{proof}
While our object of interest are trees, it is easier to work with
functions on $[n] = \{1,2,\dots,n\}$ via the Joyal mapping
(\cite{A:Joyal1081}, also presented in English in
\cite{B:Aigner_&_Ziegler2003}).

We shall briefly describe the Joyal mapping and some of its
properties that we need. The Joyal mapping maps $f$, a function from
$[n]$ to itself, to an undirected tree $T_f$ over the set of
vertices $[n]$. There are $n^n$ functions in $[n]^{[n]}$, but only
$n^{n-2}$ labeled trees over $[n]$. In order to make the mapping
into a bijection we distinguish two vertices of a labeled tree by
marking them \emph{left} and \emph{right} (we may mark one vertex
with both). Now the target set is the set of all labeled trees over
$[n]$ together with the markings, and is of size $n^n$.

The mapping is defined as follows. Let $f\colon[n]\to[n]$. Define
$\vec{G}_f$ as the functional digraph\footnote{A \emph{functional
digraph} is a directed graph with all outdegrees equal one.} with
vertex set $[n]$ and edge set $\{\dedge{i}{f(i)} | i\in[n] \}$.
Every vertex in $\vec{G}_f$ has outdegree one, so every connected
component has one directed cycle, and all edges that are not in a
cycle are pointing towards the cycle. Let $M=\{a_1 < a_2 < \dotsb <
a_m\}$ be the set of all vertices participating in a cycle of
$\vec{G}_f$. Notice that $M$ is the maximal set such that $f|_M$ is
a bijection. To get $T_f$, the tree corresponding to the function
$f$, we first define a path by taking the vertices of $M$ and adding
the $m-1$ edges of the form $\edge{f(a_i)}{f(a_{i+1})}$. We then
mark $f(a_1)$ as ``left'' and $f(a_m)$ as ``right''. Finally we add
the vertices in $[n]\setminus M$ with the edges $\edge{i}{f(i)}$
from $\vec{G}_f$ (forgetting about directions).

Given a tree $T$ with two such markings, we go back by defining
$M$ as the vertices in the path $P$ connecting ``left'' and
``right'', and directing all other vertices towards $P$. Sort the
members of $M$ according to their value and denote them by $a_1 <
a_2 < \dotsb < a_m$. We define $f$ as follows. If $i\in M$ is the
$j$'th vertex in the path then $f(i) = a_j$. If $i\notin M$ then there is one edge, $\dedge{i}{j}$,
emanating from $i$, and we set $f(i) = j$. It is easy to verify that
this is indeed the inverse of the mapping described above.

Notice that vertices that are not in a cycle are left by the Joyal
mapping as they were in $\vec{G}_f$, meaning that they will be
incident with exactly the same edges as in the functional graph. In
particular, edges with both endpoints being vertices that are not in
a cycle of $\vec{G}_f$ will touch the same edges in $T_f$ as in
$\vec{G}_f$. For our purpose, the fate of vertices lying in a cycle
is irrelevant.

Direct the edges of $R$ towards the root to get $\vec{R}$. Consider a random function
$f$ on $[n]$ and let $X$ be the random variable counting the number
of directed edges $\dedge{u}{v}$ in $\vec{G}_f$ such that\ $u,v$ and
the ancestors of $v$ in $\vec{G}_f$ do not belong to any cycle in
$\vec{G}_f$, and in addition, $v$ and its ancestors form an
isomorphic copy of $\vec{R}$.

Denote the vertices of $\vec{R}$ by $r_1,\dots,r_k$, the root being
$r_k$. Fix a $(k+1)$-tuple of vertices of $\vec{G}_f$, say
$1,2,\dots,k+1$. The probability that the edge
$\dedge{k}{k+1}$ meets the condition described above is at least the
probability that $\dedge{k}{k+1} \in E(\vec{G}_f)$, the mapping $i
\rightarrow r_i$ is an isomorphism between $\vec{R}$ and
$\vec{G}_f[\{1,\dots,k\}]$, and in addition, there are no other
edges of $\vec{G}_f$ incoming to $\{1,\dots,k+1\}$.  The latter is
\[ \left( \frac{1}{n} \right)^k \left( \frac{n-k-1}{n} \right)^{n-k}
. \] %
In order to see this simply notice that for $1 \leq i \leq k$
there is only one valid target for $f(i)$, while for $i \geq k+1$
it is enough to require that $f$ will map $i$ outside of
$\{1,2,\dots,k+1\}$. Therefore we get
\[ \E X \geq \binom{n}{k+1} n^{-k} \left( \frac{n-k-1}{n} \right)^{n-k}, \]
which implies $\E X = \Omega(n)$.

We want to show that $X$ is concentrated around its mean. Consider
the \emph{value exposing} martingale, in which we expose the values
of $f$ one by one. Now, changing the value of $f$ in one coordinate,
$i$, can ruin at most two copies of $\vec{R}$ (one using the edge
$\dedge{i}{f(i)}$ and another that now has an extra edge
$\dedge{i}{f'(i)}$). Therefore the Lipschitz condition with constant
two holds and we can apply the Azuma Inequality
\cite{B:alon_&_spencer08, A:azuma1967} which yields $\Pr[X = 0] <
e^{-cn} $ for some constant $c>0$.

Observe that if $X(f)>0$ then by the definition of $X$, the
corresponding tree $T_f$ contains the edge $\edge{u}{v}$ requested
by the proposition.

As mentioned above, the Joyal correspondence is $n^2$ to one. If a
labeled tree $T$ does not contain an edge as required, then all its
$n^2$ preimages $f$ satisfy $X(f) = 0$. Therefore, the probability
not to get a tree with a required edge is at most $\Pr[X=0] <
e^{-cn}$ as proven above.
\end{proof}

The next argument of the proof states the existence of a
\emph{nullifying tree} $Z$ (depending on $F$ and $m$) such that if a
tree $T$ has a $Z$-leaf then $s(F,T) \equiv 0 \imod{m}$.
\begin{lemma} \label{Lem: Nullifying tree, non induced}
Let $F$ be a tree with at least one edge and let $m$ be an integer.
Then there exists a rooted tree $Z$ such that, if $Z \cong
T^{(u,v)}$ for some edge $\edge{u}{v} \in T$, then $s(F,T) \equiv 0
\imod{m}$.
\end{lemma}

\begin{proof}
The proof is constructive. By Proposition \ref{Propo: Exists Y with
s(F,Y) equiv 0 mod m} to be proven below
there exists a tree $Y$ such that $s(F,Y)
\equiv 0 \imod{m}$.

Let $\Delta(F)$ be the maximal degree of $F$. To get $Z$ take
$\Delta(F) + 1$ copies of $Y$, add a new vertex $r$ to be viewed as
the root of $Z$, and connect $r$ to a vertex of each $Y$ (thus
adding $\Delta(F) + 1$ edges). 

Let $T$ be a tree and assume that $Z \cong T^{(u,v)}$ for some edge
$\edge{u}{v} \in T$. We wish to show that $s(F,T) \equiv 0
\imod{m}$. There are finitely many ways in which one may cover $v$
by a copy of $F$, and it may also be that $v$ remains uncovered. We
classify $F$-matchings in $T$ into classes $C_1,C_2,\dotsc,C_q$
according to the copy of $F$ covering $v$, with the set of
$F$-matchings not covering $v$ being a separate class $C_0$. We
argue that the number of $F$-matchings in each such class is a zero
residue. Indeed, the number of $F$-matchings in a given class $C_i$
is precisely the number of $F$-matchings in the forest remaining
from $T$ after removing $v$ and the copy covering it, if there is
one. In fact, this number is the product of the number of
$F$-matchings in every connected component of the forest. By our
construction of $Z$, at least one of the trees is this forest is
isomorphic to $Y$. Since $s(F,Y) \equiv 0 \imod{m}$ we deduce that
the number of $F$-matchings in the forest, and also in $C_i$, is
zero modulo $m$. This is true for all $C_i$, and since $S(F,T)
= \cup C_i$ one has $s(F,T) \equiv 0 \imod{m}$.
\end{proof}

Before stating and proving the next proposition we define some
notation. Let $F$ be a tree. Take a longest path in $F$ and denote its vertices by $u_1,u_2,\dots,u_{l+1}$, where $l$ is the diameter of $F$. If we disconnect all edges of the form
$\edge{u_i}{u_{i+1}}$ we get $l+1$ subtrees. Let $b_i$ be the number
of vertices in the subtree containing $u_i$. With this notation we
have $|F| = \sum_{i=1}^{l+1} b_i$. Since $b_{l+1} =1$ we may also
write $|F| = 1 + \sum_{i=1}^l b_i$. We shall use this notation in
the proof of the next proposition and in the proof of Proposition
\ref{Propo: Exists Y' with s'(F,Y') equiv 0 mod m} as well.

\begin{proposition} \label{Propo: Exists Y with s(F,Y) equiv 0 mod
m} Let $F$ be a tree with at least one edge and let $m$ be an
integer. Then there exists a rooted tree $Y$ such that $s(F,Y)
\equiv 0 \imod{m}$.
\end{proposition}
\begin{proof}
Let $W_t$ be a tree made of $t$ copies of $F$ in which we
identify the vertex $u_{l+1}$ of copy $i$ with the vertex $u_1$ of
copy $i+1$ (for $1 \leq i \leq t-1$). Let $P \subset W_t$ be the
path in $W_t$ connecting the first copy of $u_1$ to the last copy of
$u_{l+1}$, and number its vertices by $1,\dotsc,lt+1$ in the natural
order, from the copy of $u_1$ in the first copy of $F$ to the copy
of $u_{l+1}$ in the last copy of $F$. We want to have a direction of
parenthood in $W_t$, so we set $1$ to be the root. Notice that all
connected components of $W_t \setminus V[P]$ are of size strictly
less than $|F|$.

We are interested in embeddings of $F$ in $W_t$, that is, in
subgraphs of $W_t$ that are isomorphic to $F$. Notice that every
such embedding must have a vertex in $P$. Let $C$ be an embedding of
$F$ in $W_t$. We call the vertex $\min\{C \cap P\}$ the
\emph{starting vertex} of $C$. Consider the set of all starting
vertices in $W_t$. If $1 \leq i \leq (t-2)l+1$ is a starting vertex,
then by symmetry so is $i+l$. Observe that trivially $1$ is a
starting vertex (and so are $l+1, 2l+1,\dotsc$). By the symmetry
argument above, if there are $d$ starting vertices between $1$ and
$l+1$ (inclusive), then there are $1+(t-1)(d-1)$ starting vertices
in $W_t$. To see this recall that $1$ is always a starting vertex,
and each copy but the last adds $d-1$ starting vertices; also, the
last copy of $F$ in $W_t$ does not contain any starting vertices
apart from $1+l(t-1)$ as deleting $1+l(t-1)$ leaves less than $|F|$
vertices to the right of it. Similarly, if $i$ is a starting vertex
then there are $d$ starting vertices between $i$ and $i+l$,
inclusive.

Now we can define $\{Y_r\}$, a family of subtrees of $W_t$ a member
of which will eventually be the sought after tree. Set $t$ to be
large enough ($t = 1+\lceil(r-1)/(d-1)\rceil$ will do). To get $Y_r$
take the minimal subpath of $P\subset W_t$ containing the last $r$
starting vertices and then append to each vertex in the subpath the
subtree of its descendants through children outside $P$. For
example, $Y_1$ is the single starting vertex $1+l(t-1)$ and $Y_d$ is
the next to the last copy of $F$ in $W_t$.

Let $g(r)$ be the number of $F$-matchings in $Y_r$. We count such
$F$-matchings by the membership of $i$, the first vertex in $Y_r$.
If $i$ is not covered by the matching, then the next embedding of
$F$ begins no earlier than the next starting vertex. This means
that the number of $F$-matchings of $Y_r$ in which $i$ is not
covered is $g(r-1)$.

We argue now that if $i$ is covered by the matching then the next
$d-1$ starting vertices are also covered. Let $\varphi\colon F \to
Y_r$ be an embedding covering $i$. We claim that the next $d-1$
starting vertices are also covered by $\varphi$. First, since the
diameter of $F$ is $l$, no vertex of $P$ farther than $i+l$ (which
is the starting vertex $d-1$ away from $i$) is covered by $\varphi$.
On the other hand, the path from $i$ to $i+l-1$ contains one copy of
each $u_i$ (not necessarily in the natural order). Thus, the number
of vertices in the set containing $i, i+1, \dotsc, i+l-1$ and their
descendants is exactly $\sum_{i=1}^l b_i$, hence $\varphi$ extends
also to $i+l$. Therefore, the other embeddings in the $F$-matching
need to start after $i+l$. We get that the number of such matchings
is exactly $g(r-d)$. This gives the recursion $g(r) = g(r-1) +
g(r-d)$.

Observe that the tree $Y_r$, $1\le r<d$, does not contain a copy of
$F$, and thus the only $F$-matching in $Y_r$ is the empty one,
implying $g(r) = 1$ for every $1 \leq r < d$; also, $g(d) = 2$ as
$Y_d = F$. We can extend the recursion backwards by defining $g(0) =
1$ and $g(-1) = 0$. By Claim \ref{claim: rec} below there is an
integer $r_0 > 0$ such that $g(r_0) \equiv 0 \imod{m}$. Define $Y =
Y_{r_0}$. By the definition of $g(r)$ we have $s(F,Y) \equiv 0
\imod{m}$.
\end{proof}

\begin{claim} \label{claim: rec}
Let $g(r)\colon\mathbb{N}\to\mathbb{Z}$ be a sequence of integers
obeying a recurrence relation with integer coefficients $g(r) =
\sum_{i=1}^d c_i g(r-i)$. Assume that $g(0) = 0$ and $c_d = 1$.
Then for every positive integer $m>0$ there exists an index $r_0 =
r_0(m)>0$ such that $g(r_0) \equiv 0 \imod{m}$.
\end{claim}
\begin{proof}
First we claim that $g(r) \imod{m}$ is periodic. Indeed, since
$g(r) \imod{m}$ is determined by the $d$-tuple of the previous $d$
values, and since modulo $m$ there are at most $m^d$ possible
$d$-tuples, then after at most $m^d$ steps the sequence $g(r)
\imod{m}$ must become periodic. Next we claim that $g(r) \imod{m}$
is periodic from the beginning. To see this simply extend the
sequence $m^d$ steps backwards using the recurrence relation
$g(r-d) = g(r) - \sum_{i=1}^{d-1} c_i g(r-i)$. The previous argument
shows that the extended sequence is periodic starting at most at
the $m^d$'th element, which is the first element of the original
sequence. Hence $g(r) \imod{m}$ is periodic from its first
element, $g(0)=0$, and thus there is some $r_0 > 0$ such that
$g(r_0) \equiv 0 \imod{m}$.
\end{proof}

\section{The induced case}\label{sec: induced case}
In this section we prove Theorem \ref{Thm: main thm, induced}. The
proof is similar to the proof of Theorem \ref{Thm: main thm,
general} and we shall focus on the differences between the proofs.
As before, the proof is probabilistic. Lemma \ref{Lem: aas has R
cong T_n^(u,v)} is the probabilistic part here as well, but the
deterministic part is replaced by Lemma \ref{Lem: Nullifying tree,
induced} below.

We begin by constructing a nullifying rooted tree from copies of a
tree $Y'$ having $s'(F,Y') \equiv 0 \imod{m}$.
\begin{lemma} \label{Lem: Nullifying tree, induced}
Let $F$ be a tree and let $m$ be an integer. There exists a rooted
tree $Z'$ such that if $Z' \cong T^{(u,v)}$ for some edge
$\edge{u}{v} \in T$, then $s'(F,T) \equiv 0 \imod{m}$.
\end{lemma}

\begin{proof}
By Proposition \ref{Propo: Exists Y' with s'(F,Y') equiv 0 mod m}
below there exists a tree $Y'$ such that $s'(F,Y') \equiv 0
\imod{m}$. Construct $Z'$ by taking $\Delta(F) + 2$ copies of $Y'$,
adding a new vertex $r$ to be viewed as the root of $Z$', connecting
one copy to $r$ with a new edge and connecting the rest of the
$\Delta(F) + 1$ copies to $r$ via a path of length two.

Let $T$ be a tree and assume that $Z' \cong T^{(u,v)}$ for some edge
$\edge{u}{v} \in T$. We need to show that $s'(F,T) \equiv 0
\imod{m}$.

There are finitely many ways in which $v$ may be covered by a copy
of $F$, if it is covered at all. We classify induced $F$-matchings
according to the copy of $F$ covering $v$. Denote these classes by
$C_1,\dotsc,C_k$ and let $C_0$ be the class of all induced
$F$-matchings of $T$ in which $v$ is left uncovered. Clearly
$S'(F,T) = \bigcup_{i=0}^k C_i$. We claim
that $|C_i| \equiv 0 \imod{m}$ for every $0\leq i\leq k$.

Consider first the class $C_0$ of induced $F$-matchings in $T$ that
leave $v$ uncovered. The number of such matchings is the number of
matchings in the forest remaining after deleting $v$.
This forest has a component isomorphic to $Y$ --- the copy of $Y$
that was connected to $v$ by a single edge. The number of
induced $F$-matchings in $C_0$ is then the product of the number of
induced $F$-matchings in every connected component of the
aforementioned forest which is zero modulo $m$.

Consider now the class $C_i$ for $i>0$. As before, there is a
natural one to one correspondence between induced $F$-matchings in
$T$ that belong to $C_i$ and induced $F$-matchings of the forest
remaining after removing the copy of $F$ covering $v$ and all neighbors of vertices in that copy. Since $v$ is covered by the matching, all of its neighbors that are not covered
by the same copy of $F$ must remain uncovered. Otherwise, an
additional edge outside the copies of $F$ would be spanned. This
means that in the above forest at least one of the $\Delta(F)+1$ copies that were
connected to $v$ by a path of length two will now remain as a
connected component. Hence, the number of induced $F$-matchings in
$C_i$ is a zero residue.

Summing the sizes of the $C_i$'s we get that $m'(F,T) \equiv 0 \imod{m}$.
\end{proof}

\begin{proposition} \label{Propo: Exists Y' with s'(F,Y') equiv 0 mod m}
Let $F$ be a tree and let $m$ be an integer. Then there exists a
rooted tree $Y'$ such that $s'(F,Y') \equiv 0 \imod{m}$.
\end{proposition}
\begin{proof}
The construction and the proof are similar to those in the proof
of Proposition
\ref{Propo: Exists Y with s(F,Y) equiv 0 mod m}, and we
shall use the notation defined just before it. We define $W'_t$ as a
collection of $t$ disjoint copies of $F$, and we add an edge between
the vertex $u_{l+1}$ of the $i$'th copy and the vertex $u_1$ of the
$(i+1)$'th copy. We think of the first copy of $u_1$ as the root of
$W'_t$.

Let $P'$ be the path connecting the first copy of $u_1$ with the
last copy of $u_{l+1}$ and denote its vertices by $1,\dotsc,t(l+1)$
in the natural order. We define starting vertices in the same manner
as in the proof of Lemma \ref{Lem: Nullifying tree, non induced}.
The symmetry argument still holds, only now the period is $l+1$,
that is, if $1\le i\le (t-2)(l+1)+1$ is a starting vertex then so is
$i+l+1$. Also, if there are $d$ starting vertices between $1$ and
$l+1$, then there are $d$ starting vertices between every starting
vertex $i$ and $i+l$ and all in all there are $(t-1)d+1$ starting
vertices in $W'_t$.

Let $Y'_r$ be the subgraph of $W'_t$ composed of the minimal path of
$P$ containing the last $r$ starting vertices together with their
descendants through vertices that are not in $P$. Hence, $Y'_1$ is a
single vertex and $Y'_{d+1}$ is a copy of $F$ with an extra vertex
connected to $u_{l+1}$. Finally we define $g'(r)$ as the number of
induced $F$-matchings in $Y'_r$.

We wish to derive a recurrence formula for $g'(r)$. We count
induced $F$-matchings of $Y'_r$ by the membership of the first
vertex. The number of induced $F$-matchings that do not cover the
first vertex (who is also the first starting vertex) is exactly
$g'(r-1)$.

Consider matchings in which the first starting vertex $i$ is
covered. The embedding of $F$ covering $i$ can not cover vertices of
$P$ farther than $i+l$, since the diameter of $F$ is $l$. On the
other hand, the number of vertices in the subgraph made of the path
connecting $i$ to $i+l$ together with their descendants that are not
in $P$ is exactly $\sum b_i = |F|$. Hence $i+l$ is also covered by
the same embedding that covers $i$. Now, if $i+l+1$ is covered by
another embedding of $F$, then $\edge{i+l}{i+l+1}$ is spanned, which
is forbidden, so $i+l+1$ is not covered. Since there are $d$
starting vertices between $i$ and $i+l$, and since $i+l+1$ is a
starting vertex as well, we get that the number of such matchings is
exactly $g'(r-d-1)$. Therefore we have $g'(r) = g'(r-1) +
g'(r-d-1)$.

Clearly $g'(r) = 1$ for every $1 \leq r \leq d-1$, as the number of
vertices in $Y'_r$ in these cases is smaller than $|F|$. The value
of $g'(d)$ may be either $1$ or $2$, depending on whether $F$ may be
embedded into $Y_d$ or not. The value of $g'(d+1)$ can also be one
of a few options. Still, we extend $g'$ backwards by defining $g'(0)
= g'(d+1) - g'(d)$, $g'(-1) = g'(d) - g'(d-1)$, and $g'(-2) =
g'(d-1) - g'(d-2) = 0$. We complete the proof by applying Claim
\ref{claim: rec}.
\end{proof}

\section{Concluding discussion}
Our initial objective was to provide a simple and intuitive
explanation to the fact that almost all labeled trees have an even
number of independent sets. Indeed, there are nullifying trees $Z$
s.t.\ when a tree $T$ has a $Z$-leaf, the number of independent sets
in $T$ is even. Also, every fixed tree $Z$ appears as a $Z$-leaf in
a random tree with $n$ vertices with probability tending to one as
$n$ goes to infinity. Therefore almost all trees have an even number
of independent sets.

The simplicity of the explanation allowed vast generalizations ---
Theorems \ref{Thm: main thm, general} and \ref{Thm: main thm,
induced} above. In fact, the proof also works in other scenarios. If
a probability space of graphs has a property corresponding to the
probabilistic part of the proof, then the number of (induced)
$F$-matchings will be a zero residue in that probability space as
well.

As a concrete example, let $\mathcal{P}_n$ be the \emph{random
planar graph} of order $n$, that is, $\mathcal{P}_n$ is the set of
all simple labeled planar graphs with $n$ vertices endowed with the
uniform distribution. In \cite{A:mcdiarmid_&_steger_&_welsh2005} it
is shown that with probability exponentially close to one,
$\mathcal{P}_n$ has an $R$-leaf for every fixed rooted tree $R$.
Thus, by the above, the number of (induced) $F$-matchings is a zero
residue in a random planar graph. Notice that $\mathcal{P}_n$ is
connected with probability at least $1/e$ as shown in
\cite{A:mcdiarmid_&_steger_&_welsh2005}, so a potential simpler
strategy of proving the same result --- showing the existence of a
component having a zero residue number of (induced) $F$-matchings
--- will not suffice.

Similar results may be obtained for other random graphs models as
well. On the other hand, if we consider dense random graphs then a
different approach is required. For example, it is not clear how the
number of independent sets typically behaves as a residue for the
binomial random graph $G(n,1/2)$. Moreover, it is not difficult to
show that for $p=p(n)$ close to $1$ in the range in which the
maximum independent set of $G(n,p)$ is $\Theta(1)
>1$ asymptotically almost surely, the number of independent sets in
$G(n,p)$ is nearly uniformly distributed modulo any constant $m$.
See \cite{Sc2010} for several related results.

Our proof implies that the number of $F$-matchings in a random tree
of order $n$ is typically zero modulo any constant $m$ when
the size of $F$ grows slowly enough with $n$. It may be
interesting to find the maximal rate of growth for which this
property still holds.

\comment{The same holds for random planar graphs with a given number
of edges, $\mathcal{P}_{n,m=dn}$, by
\cite{IP:gerke_&_mcdiarmid_&_steger_&_weissl2005}; the random planar
graph process (with $m$ edges considered or accepted) by
\cite{A:gerke_&_schlatter_&_steger_&_taraz2007}; for the binomial
random graph $G(n,p)$ before the minimal degree becomes two by
common knowledge, and so on.}
\vspace{0.3cm}

\noindent
{\bf  Acknowledgment} We thank Alan Frieze for suggesting the use of the Joyal Correspondence in the proof of Lemma \ref{Lem: aas has R cong T_n^(u,v)}.

\bibliographystyle{abbrv}

\end{document}